\documentclass[12pt]{article}
\usepackage[pdftex]{color, graphicx}
\usepackage{setspace}
\usepackage[utf8]{inputenc}
\usepackage{amsmath}
\usepackage[pdftex,bookmarks,colorlinks]{hyperref}
\usepackage{latexsym}
\usepackage{amsfonts}
\usepackage{amssymb}
\usepackage{amsthm}
\usepackage{paralist}
\usepackage{mathrsfs}

\begin{document}
\newtheorem{thm}{Theorem}
\newtheorem{cor}[thm]{Corollary}
\newtheorem{lem}{Lemma}
\theoremstyle{remark}\newtheorem{rem}{Remark}
\theoremstyle{definition}\newtheorem{defn}{Definition}

\title{Complete convergence of the Hilbert transform}
\author{Sakin Demir}
\author{Sakin Demir\\
Agri Ibrahim Cecen University\\ 
Faculty of Education\\
Department of Basic Education\\
04100 A\u{g}r{\i}, Turkey\\
E-mail: sakin.demir@gmail.com
}


\maketitle


\renewcommand{\thefootnote}{}



\renewcommand{\thefootnote}{\arabic{footnote}}
\setcounter{footnote}{0}
\maketitle
\begin{abstract}Suppose that $\{a_j\}\in \ell^1$, and suppose that for any sequence $(t_n)$ of integers there exits a constant $C_1>0$ such that
$$\sharp\left\{k\in\mathbb{Z}:\sup_{n\geq 1}\left|\sum_{i\in \mathcal{B}_n-t_n}\!\!\!\!\!\!\text{\raisebox{1.9ex}{\scriptsize$\prime$}}\;\frac{a_{k+i}}{i}\right|>\lambda\right\}\\
\leq C_1\sharp\left\{k\in\mathbb{Z}:\sup_{n\geq 1}\left|\sum_{i\in \mathcal{B}_n}\!\!\text{\raisebox{1.9ex}{\scriptsize$\prime$}}\;\frac{a_{k+i}}{i}\right|>\lambda\right\},$$
 for all $\lambda >0$, where $\mathcal{B}_n=\{-n, -(n-1), -(n-2),\dots , n-2, n-1, n\}$. Then there is a constant $C_2>0$ which does not depend on the sequence $\{a_j\}$ such that 
$$\sum_{n=1}^\infty\sharp\left\{k\in\mathbb{Z}:\left|\sum_{i=-n}^{n}\!\!\!\text{\raisebox{1.9ex}{\scriptsize$\prime$}}\;\frac{a_{k+i}}{i}\right|>\lambda\right\}\leq\frac{C_2}{\lambda}\sum_{i=-\infty}^{\infty}|a_i|$$
for all $\lambda>0$.\\
\indent
 Let $(X,\mathscr{B},\mu )$ be a measure space, $\tau :X\to X$ an invertible measure-preserving transformation, and suppose that $f\in L^1(X)$ such that for any sequence $(t_n)$ of integers there exists a constant $C_1>0$ such that
$$\mu\left\{ x: \sup_{n\geq 1}\left|\sum_{i\in \mathcal{B}_n-t_n}\!\!\!\!\!\!\text{\raisebox{1.9ex}{\scriptsize$\prime$}}\;\frac{f(\tau^ix)}{i}\right| >\lambda \right\}\leq C_1\mu\left\{x: \sup_{n\geq 1}\left|\sum_{i\in \mathcal{B}_n}\!\!\text{\raisebox{1.9ex}{\scriptsize$\prime$}}\;\frac{f(\tau^i x)}{i}\right|>\lambda \right\}  $$
for all $\lambda >0$, where $\mathcal{B}_n=\{-n, -(n-1), -(n-2),\dots , n-2, n-1, n\}$. Then
there exists a constant $C_2>0$ which does not depend on $f$ such that

$$\sum_{n=1}^\infty\mu\left\{x:\left|\sum_{i=-n}^{n}\!\!\!\,\text{\raisebox{1.9ex}{\scriptsize$\prime$}}\;\frac{f(\tau^ix)}{i}\right|>\lambda\right\}\leq\frac{C_2}{\lambda}\|f\|_1$$
for  all $\lambda >0$.
\end{abstract}


\section{Introduction}

Let $(X, \mu )$ be a measure space. A sequence of functions $(f_n)$ defined on $X$ is said to be converge completely to a constant $C$ if
$$\sum_{n=1}^\infty \mu\{x: |f_n(x)-C|>\epsilon \}<\infty$$
for all $\epsilon >0$.\\
\noindent
Since we have
\begin{align*}
\mu\left\{x:\left|\sum_{n=1}^\infty f_n(x)-C\right|>\epsilon \right\}&=\mu \bigcup_{n=1}^\infty \left\{x:\left|\sum_{k=1}^nf_k(x)-C\right|>\epsilon \right\}\\
& \leq \sum_{n=1}^\infty \mu\left\{x: \left|\sum_{k=1}^nf_k(x)-C\right|>\epsilon \right\},
\end{align*}
complete convergence of the partial summation
$$\sum_{k=1}^nf_k(x)$$
to $C$ implies the almost everywhere convergence of the series
$$\sum_{n=1}^\infty f_n(x)$$
to $C$.\\
The study of complete convergence first started in probability theory. P. L.~Hsu and H.~Robbins~\cite{hrrh} have proved the following theorem:
\begin{thm}\label{hsur} Let $(X_n)$ be a sequence of independent random variables defined on a probability space $(\Omega, \mathbb{P} )$, with the same distribution function $F(x)$ and such that
$$\int_{-\infty}^{\infty}x\, dF(x)=0,\;\;\; \int_{-\infty}^{\infty}x^2\, dF(x)<\infty .$$
Then, the sequence 
$$Y_n=\frac{1}{n}\sum_{k=1}^nX_k$$
converges to $0$ completely; i.e., the series
$$\sum_{n=1}^\infty \mathbb{P}\{|Y_n|>\epsilon \}$$
converges for every $\epsilon >0$.
\end{thm}
\noindent
P.~Erd\"os~ has first proved the converse of Theorem~\ref{hsur} with some restriction in \cite{pe}, and in a  little while he has come up with a proof  with no restriction in \cite{pe2}.\\
Together with its converse Theorem~\ref{hsur} is known as Hsu-Robbins-Erd\"os Theorem. Later complete convergence  has been studied with different type of settings  by some other authors in probability theory, see, for example,  L. E.~Baum and M.~Katz~\cite{bkatz}. Our goal is to carry the  notion of complete convergence to harmonic analysis and ergodic theory. We prove that the discrete and ergodic Hilbert transforms completely converge to 0.  Note that those transfer principles used to transfer the results for discrete averages to the continuous case can be used to find the analogues of our results for continuous Hilbert transform both on the real line and on a dynamical system equipped with an ergodic measure preserving flow  since the Hilbert transform is an operator of convolution type.

\section{Preliminaries}
Let $(X,\mathscr{B},\mu )$ be a measure space, $\tau :X\to X$ an invertible measure-preserving transformation. The ergodic Hilbert transform of a measurable function $f$, is defined as
$$Hf(x)=\lim_{n\to\infty}\sum_{k=-n}^{n}\:\!\!\!\!\text{\raisebox{1.9ex}{\scriptsize$\prime$}}\;\frac{f(\tau^kx)}{k}.$$
The prime denotes that the term with zero denominator is omitted in the summation.\\
\noindent
It is well known that $Hf$ is of weak type $(p,p)$ for $1\leq p<\infty$, and of strong type $(p,p)$ for $1<p<\infty$. There are several different methods in the literature  to see these facts. The most immediate one is to transfer the same inequalities for the Hilbert transform on $\mathbb{R}$ by Calder\'on transfer principle as in the relation between the Hardy-Littlewood maximal function and the ergodic maximal function.\\
\noindent
For $\{a_j\}\in l^1$ the Hilbert transform on $\mathbb{Z}$ is defined by
$$\mathcal{H}a(k)=\lim_{n\to\infty}\sum_{i=-n}^{n}\:\!\!\!\!\text{\raisebox{1.9ex}{\scriptsize$\prime$}}\;\frac{a_{k+i}}{i}.$$
Our main goal of this research is to prove the following:\\
Suppose that $\{a_j\}\in l^1$ has finite support. Then we prove that there is a constant $C$ such that 
$$\sum_{n=1}^\infty\sharp\left\{k\in\mathbb{Z}:\left|\sum_{i=-n}^{n}\!\!\!\text{\raisebox{1.9ex}{\scriptsize$\prime$}}\;\frac{a_{k+i}}{i}\right|>\lambda\right\}\leq\frac{C}{\lambda}\sum_{i=-\infty}^{\infty}|a_i|$$
for all $\lambda>0$. Then it will be clear by means of a transference argument that the same type of inequality for the ergodic Hilbert transform also remains true.
The following lemmas are due to  L. H. Loomis~\cite{lhloom}, who rediscovered an idea that essentially goes back to  G. Boole~\cite{gboole}. We give the proofs of them for completeness:
\begin{lem}\label{weakinlemgs} Let $a_1,a_2,\dots ,a_n\geq 0$ and $g(s)=\sum_{i=1}^n\frac{a_i}{s-t_i}$. Then
$$m\{s:g(s)>\lambda\}=m\{s:g(s)<-\lambda\}=\frac{1}{\lambda}\sum_{i=1}^na_i,$$
where $m$ denotes the Lebesgue measure on $\mathbb{R}$.
\end{lem}
\begin{proof} Since $g(t_i-)=-\infty$, $g(t_i+)=\infty$ and $g^{\prime}(s)<0$ for all $s$, there are precisely $n$ points $m_i$ such that $g(m_i)=\lambda$, and $t_i<m-i<t_{i+1}$, $i=1,2,\dots ,n-1,t_n,m_n$. The set where $g(s)>\lambda$ thus consists of the intervals $(t_i,m_i)$ and has total length
\begin{equation}\label{totallength}
\sum_{i=1}^n(m_i-t_i)=\sum_{i=1}^nm_i-\sum_{i=1}^nt_i.
\end{equation}
But the numbers $m_i$ are the roots of the equation
$$\sum_{i=1}^n\frac{a_i}{s-t_i}=\lambda ,$$
whose cross-multiplied form is
$$\sum_{i=1}^na_i\left[\prod_{j\neq i}(s-t_i)\right]=\lambda\prod_{i=1}^n(s-t_i),$$
or
$$\lambda s^n-\left[\lambda\sum t_j+\sum a_i\right]s^{n-1}+\dots =0,$$
so that
\begin{equation}\label{sumofmi}
\sum_{i=1}^nm_i=\sum_{i=1}^nt_i+\frac{1}{\lambda}\sum_{i=1}^na_i.
\end{equation}
The first part of the lemma follows from (\ref{totallength}) and (\ref{sumofmi}); the proof for $g(s)<-\lambda$ is almost identical.
\end{proof}
\begin{lem}\label{weakinleminthil} There is a constant $C$ such that if $\{a_k\}\in \ell^1$ and $\lambda >0$, then
$$\sharp\left\{k\in\mathbb{Z}:\left|\sum_{i=-\infty}^{\infty}\!\!\!\text{\raisebox{1.9ex}{\scriptsize$\prime$}}\;\frac{a_{k+i}}{i}\right|>\lambda\right\}\leq\frac{C}{\lambda}\sum_{i=-\infty}^{\infty}|a_i|.$$
\end{lem}
\begin{proof} By treating the positive and negative ones separately, we may assume that all the $a_i$ are positive. We will count
$$A_{\lambda}=\left\{k:\sum_{i=-\infty}^{\infty}\!\!\!\text{\raisebox{1.9ex}{\scriptsize$\prime$}}\;\frac{a_{k+i}}{i}>\lambda\right\};$$
a similar method will apply to
$$A_{\lambda}^{\prime}=\left\{k:\sum_{i=-\infty}^{\infty}\!\!\!\text{\raisebox{1.9ex}{\scriptsize$\prime$}}\;\frac{a_{k+i}}{i}<-\lambda\right\}.$$
Choose a finite set $A\subset A_{\lambda}$, and choose $N$ so large that $A\subset \left[N,N\right]$ and, for each $k\in A$, 
$$\sum_{i=-N}^N\!\!\!\text{\raisebox{1.9ex}{\scriptsize$\prime$}}\;\frac{a_i}{i-k}>\lambda.$$
Then
$$g_k(s)=\sum_{i=-N}^N\!\!\!\!\:\text{\raisebox{1.9ex}{\scriptsize$\prime$}}\;\frac{a_i}{i-s}>\lambda$$
for $s=k\in A$, and hence $g_k(s)>\lambda$ for $s\in [k,k+1)$, because $g_k^{\prime}(s)>0$. If we let
$$g(s)=\sum_{i=-N}^N\!\!\!\!\:\text{\raisebox{1.9ex}{\scriptsize$\prime$}}\;\frac{a_i}{i-s}>\lambda$$
and
$$h_k(s)=\frac{a_k}{k-s},$$
then $g=g_k+h_k$, so that for each $k\in A$
$$(k,k+1)\subset\left\{s:g_k(s)>\lambda\right\}\subset\left\{s:g(s)>\frac{1}{\lambda}\right\}\cup\left\{s:h_k(s)<-\frac{\lambda}{2}\right\}.$$
Therefore, we get
\begin{align*}
\sharp{A}&=m\left(\bigcup_{k\in A}(k,k+1)\right)\\
&\leq m\left\{s:g(s)>\frac{\lambda}{2}\right\}+\sum_{k\in A}m\left\{s:h_k(s)<-\frac{\lambda}{2}\right\}\\
&\leq\frac{2C}{\lambda}\sum_{i=-N}^Na_i+\sum_{k\in A}\frac{2C}{\lambda}a_k\\
&\leq\frac{4C}{\lambda}\|a\|_1
\end{align*}
as desired.
\end{proof}
\begin{lem}\label{weakinlemintmaxhil} There is a constant $C$ such that if $\{a_k\}\in \ell^1$ and $\lambda >0$, then
$$\sharp\left\{k\in\mathbb{Z}:\sup_{n\geq 1}\left|\sum_{i=-n}^{n}\!\!\!\text{\raisebox{1.9ex}{\scriptsize$\prime$}}\;\frac{a_{k+i}}{i}\right|>\lambda\right\}\leq\frac{C}{\lambda}\sum_{i=-\infty}^{\infty}|a_i|.$$
\end{lem}
\begin{proof} We assume as before that all the $a_i$ are positive and drop the absolute value signs. Let
$$A\subset\left\{k:\sup_{n\geq 1}\sum_{i=-n}^{n}\!\!\!\text{\raisebox{1.9ex}{\scriptsize$\prime$}}\;\frac{a_{k+i}}{i}>\lambda\right\}$$
be closed and bounded. For each $k\in A$ there is an interval of integers $I_k=\left[k-n_k,k+n_k\right]$ such that
$$\sum_{i\in I_k}\!\text{\raisebox{1.9ex}{\scriptsize$\prime$}}\;\frac{a_i}{i-k}>\lambda .$$
Let
$$g_k(s)=\sum_{i\in I_k}\!\text{\raisebox{1.9ex}{\scriptsize$\prime$}}\;\frac{a_i}{i-s},\quad g(s)=\sum_{i=-\infty}^{\infty}\!\!\!\text{\raisebox{1.9ex}{\scriptsize$\prime$}}\;\frac{a_i}{i-s},\quad h_k(s)=\sum_{i\notin I_k}\!\text{\raisebox{1.9ex}{\scriptsize$\prime$}}\;\frac{a_i}{i-s}.$$
If $k\in A$, then $g_k(k)>\lambda$, so that either $g(k)>\frac{\lambda}{2}$ or $h_k(k)<-\frac{\lambda}{2}$. In the first case ($k\in A_1$), by Lemma~\ref{weakinleminthil}, $k$ falls into a single (independent of $k$) set of measure no more than $\frac{C}{\lambda}\|a\|_1$. To deal with the left over $k$'s ($k\in A_2$), replace $\{I_k\}$ by a disjoint subfamily which still covers at least $\frac{1}{3}$ of $A_2$, by at each stage selecting an interval of maximal disjoint from the previously chosen ones. Find $N$ such that
$$\bigcup_{k\in A_2}I_k\subset\left[-N,N\right]$$
and
$$\tilde{h}_k(k)\leq-\frac{\lambda}{2}\;\textrm{for all}\;k\in A_2,$$
where
$$\tilde{h}_k(s)=\sum_{i\in\{-N,\dots ,N\}-I_k}\frac{a_i}{i-s}.$$
Then also $\tilde{h}_k(s)<-\frac{\lambda}{2}$ on $(k-n_k,k)$, so that we find
\begin{align*}
\sharp{A_1}&=\sharp{A_2}+\sharp{A_2}\\
&\leq\frac{C}{\lambda}\|a\|_1+6\sum_{k\in A_2}n_k\\
&\leq\frac{C}{\lambda}\|a\|_1+6m\left(\bigcup_{k\in A_2}\left\{s:\tilde{h}_k(s)<-\frac{\lambda}{2}\right\}\right)\\
&\leq\frac{C}{\lambda}\|a\|_1+6m\left(\bigcup_{k\in A_2}\left(\left\{s:\sum_{i=-N}^N\:\!\!\!\!\text{\raisebox{1.9ex}{\scriptsize$\prime$}}\;\frac{a_i}{i-s}<-\frac{\lambda}{4}\right\}\cup\left\{s:g_k(s)>\frac{\lambda}{4}\right\}\right)\right)\\
&\leq\frac{C}{\lambda}\|a\|_1+6m\left\{s:\sum_{i=-N}^N\:\!\!\!\!\text{\raisebox{1.9ex}{\scriptsize$\prime$}}\;\frac{a_i}{i-s}<-\frac{\lambda}{4}\right\}\cup\left\{s:g_k(s)>\frac{\lambda}{4}\right\}+6\sum_{k\in A_2}m\left\{s:g_k(s)>\frac{\lambda}{4}\right\}\\
&\leq\frac{C}{\lambda}\|a\|_1+\frac{24C}{\lambda}\|a\|_1+6\sum_{k\in A_2}\frac{4C}{\lambda}\sum_{i\in I_k}a_i\\
&\leq\frac{49C}{\lambda}\|a\|_1.
\end{align*}
\end{proof}
\section{The Results}
The following is our first result:
\begin{thm}\label{mainthm}Suppose that $\{a_j\}\in \ell^1$, and suppose that for any sequence $(t_n)$ of integers there exits a constant $C_1>0$ such that
$$\sharp\left\{k\in\mathbb{Z}:\sup_{n\geq 1}\left|\sum_{i\in \mathcal{B}_n-t_n}\!\!\!\!\!\!\text{\raisebox{1.9ex}{\scriptsize$\prime$}}\;\frac{a_{k+i}}{i}\right|>\lambda\right\}\\
\leq C_1\sharp\left\{k\in\mathbb{Z}:\sup_{n\geq 1}\left|\sum_{i\in \mathcal{B}_n}\!\!\text{\raisebox{1.9ex}{\scriptsize$\prime$}}\;\frac{a_{k+i}}{i}\right|>\lambda\right\},$$
 for all $\lambda >0$, where $\mathcal{B}_n=\{-n, -(n-1), -(n-2),\dots , n-2, n-1, n\}$. Then there is a constant $C_2>0$ which does not depend on the sequence $\{a_j\}$ such that 
$$\sum_{n=1}^\infty\sharp\left\{k\in\mathbb{Z}:\left|\sum_{i=-n}^{n}\!\!\!\text{\raisebox{1.9ex}{\scriptsize$\prime$}}\;\frac{a_{k+i}}{i}\right|>\lambda\right\}\leq\frac{C_2}{\lambda}\sum_{i=-\infty}^{\infty}|a_i|$$
for all $\lambda>0$.
\end{thm}
\begin{proof}Let us first define the integer block $\mathcal{B}_n=\{-n, -(n-1), -(n-2),\dots , n-2, n-1, n\}$ for each $n\in\mathbb{Z}$.  Let
$$\mathcal{A}_n=\left\{k\in\mathbb{Z}:\left|\sum_{i=-n}^{n}\!\!\!\text{\raisebox{1.9ex}{\scriptsize$\prime$}}\;\frac{a_{k+i}}{i}\right|>\lambda\right\}$$
and
$$\mathcal{A}=\left\{k\in\mathbb{Z}:\sup_{n\geq 1}\left|\sum_{i=-n}^{n}\!\!\!\text{\raisebox{1.9ex}{\scriptsize$\prime$}}\;\frac{a_{k+i}}{i}\right|>\lambda\right\}.$$
Then we have
$$\mathcal{A}_n\subset\mathcal{A}\;\;\textrm{for all}\;\; n\geq 1.$$
This imples that $\sharp\mathcal{A}_n\leq\sharp\mathcal{A}$ for all $n\geq 1$ and since $\sharp\mathcal{A}<\infty$ by Lemma~\ref{weakinlemintmaxhil} we see that $\sharp\mathcal{A}_n<\infty$ for all $n\geq 1$. This shows that $\mathcal{A}_n$ has finitely many elements for all $n\geq 1$ since $\sharp$ is the counting measure on $\mathbb{Z}$, and thus $\mathcal{A}_n$ is a bounded set for each $n\geq 1$.

Since $A_n$ is bounded, we can inductively select a sequence $t_n$ so that the translates $A_n-t_n$ are pairwise disjoint . 
 Note that the $A_n$ are intervals.  Move $A_n$ far away, outside of the union of $A_k-t_k$,  $k=1,\dots ,n-1$, and this way can have

$$(\mathcal{A}_n-t_n)\cap (\mathcal{A}_{n^\prime}-t_{n^\prime})=\phi\; \textrm{if}\;n\neq n^\prime .$$

$$\sharp (\mathcal{A}_n-t_n)=\sharp \mathcal{A}_n$$
we only need to prove that

$$\sum_{n=1}^\infty\sharp (\mathcal{A}_n-t_n)\leq\frac{C}{\lambda}\sum_{i=-\infty}^{\infty}|a_i|$$
for some constant $C$.\\
We now have
\begin{align*}
\sum_{n=1}^\infty\sharp (\mathcal{A}_n-t_n)&=\sum_{n=1}^\infty\sharp\left\{k\in\mathbb{Z}:\left|\sum_{i\in \mathcal{B}_n-t_n}\!\!\!\!\!\!\text{\raisebox{1.9ex}{\scriptsize$\prime$}}\;\frac{a_{k+i}}{i}\right|>\lambda\right\}\\
&=\sharp\bigcup_{n=1}^\infty \left\{k\in\mathbb{Z}:\left|\sum_{i\in \mathcal{B}_n-t_n}\!\!\!\!\!\!\text{\raisebox{1.9ex}{\scriptsize$\prime$}}\;\frac{a_{k+i}}{i}\right|>\lambda\right\}\\
&\leq \sharp\left\{k\in\mathbb{Z}:\sup_{n\geq 1}\left|\sum_{i\in \mathcal{B}_n-t_n}\!\!\!\!\!\!\text{\raisebox{1.9ex}{\scriptsize$\prime$}}\;\frac{a_{k+i}}{i}\right|>\lambda\right\}\\
&\leq C_1\sharp\left\{k\in\mathbb{Z}:\sup_{n\geq 1}\left|\sum_{i\in \mathcal{B}_n}\!\!\text{\raisebox{1.9ex}{\scriptsize$\prime$}}\;\frac{a_{k+i}}{i}\right|>\lambda\right\}\\
&\leq\frac{C_2}{\lambda}\sum_{i=-\infty}^{\infty}|a_i|      \;\;\;(\textrm{by Lemma~\ref{weakinlemintmaxhil}})
\end{align*}
as desired.
\end{proof}
Our second result is the following:
\begin{cor}
 Let $(X,\mathscr{B},\mu )$ be a measure space, $\tau :X\to X$ an invertible measure-preserving transformation, and suppose that $f\in L^1(X)$ such that for any sequence $(t_n)$ of integers there exists a constant $C_1>0$ such that
$$\mu\left\{ x: \sup_{n\geq 1}\left|\sum_{i\in \mathcal{B}_n-t_n}\!\!\!\!\!\!\text{\raisebox{1.9ex}{\scriptsize$\prime$}}\;\frac{f(\tau^ix)}{i}\right| >\lambda \right\}\leq C_1\mu\left\{x: \sup_{n\geq 1}\left|\sum_{i\in \mathcal{B}_n}\!\!\text{\raisebox{1.9ex}{\scriptsize$\prime$}}\;\frac{f(\tau^i x)}{i}\right|>\lambda \right\}  $$
for all $\lambda >0$, where $\mathcal{B}_n=\{-n, -(n-1), -(n-2),\dots , n-2, n-1, n\}$. Then
there exists a constant $C_2>0$ which does not depend on $f$ such that

$$\sum_{n=1}^\infty\mu\left\{x:\left|\sum_{i=-n}^{n}\!\!\!\,\text{\raisebox{1.9ex}{\scriptsize$\prime$}}\;\frac{f(\tau^ix)}{i}\right|>\lambda\right\}\leq\frac{C_2}{\lambda}\|f\|_1$$
for  all $\lambda >0$.
\end{cor}
\begin{proof} The transference argument we are about use to proof our Corollary is the modification of the proof of Lemma 1 in K. Petersen~\cite{kpetersen} to our case. One can also directly apply a well known variant of the transfer principle of A. P.~Calder\'on~\cite{apcal} to Theorem~\ref{mainthm} to get the desired result.\\

 By considering $f^+$ and $f^-$ separately, we may assume that $f\geq 0$. We will show that
$$\sum_{n=1}^\infty\mu\left\{x:\left|\sum_{i=-n}^{n}\!\!\!\,\text{\raisebox{1.9ex}{\scriptsize$\prime$}}\;\frac{f(\tau^ix)}{i}\right|>\lambda\right\}\leq\frac{C}{\lambda}\|f\|_1,$$
where $C$ is a constant independent of $f$ and $\lambda$.

For fixed $x$ and $K$, let $a_k=f(\tau^kx)$ and
$$
a_k^K = \left\{ \begin{array}{ll}
a_k & \textrm{if $|k|\leq K$},\\
0 & \textrm{if $|k|>K$},
\end{array} \right.
$$
so that $\{a_k^K\}\in l^1$. For each $j\in\mathbb{Z}$, let
$$G_j(x)=\left|\sum_{k=-n}^n\:\!\!\!\!\text{\raisebox{1.9ex}{\scriptsize$\prime$}}\;\frac{a_{k+j}}{k}\right|,\quad\textrm{and}\quad G_j^K(x)=\left|\sum_{k=-n}^n\:\!\!\!\!\text{\raisebox{1.9ex}{\scriptsize$\prime$}}\;\frac{a_{k+j}^K}{k}\right|.$$

Then
\begin{align*} 
G_j(x)&=\left|\sum_{k=-n}^n\:\!\!\!\!\text{\raisebox{1.9ex}{\scriptsize$\prime$}}\;\frac{a_{k+j}^K}{k}+\frac{a_{k+j}-a_{k+j}^K}{k}\right|\\
&\leq G_j^K(x)+\left|\sum_{k=-n}^n\:\!\!\!\!\text{\raisebox{1.9ex}{\scriptsize$\prime$}}\;\frac{a_{k+j}-a_{k+j}^K}{k}\right|,
\end{align*}
so that $G_j(x)\leq G_j^K(x)$ for $|j|\leq K$.

Now let $E=\{x:G_0(x)>\lambda\}$, so that $\{x:G_j(x)>\lambda\}=\tau^{-j}E$. Let $\bar{E}=\left\{(x,j):G_j^K(x)>\lambda\right\}$. Then, if $\sharp$ continues to denote the counting measure on $\mathbb{Z}$,
\begin{align*}
\sum_{n=1}^\infty\mu\times\sharp{(\bar{E})}&=\int_X\sum_{n=1}^\infty\sharp\left\{j:G_j^K(x)>\lambda\right\}\,d\mu (x)\\
&\leq\int_X\frac{C}{\lambda}\sum_{j=-\infty}^{\infty}\left|a_j^K\right|\,d\mu\\
&\leq\int_X\frac{C}{\lambda}\sum_{-K}^{K}|a_j|\,d\mu\\
&\leq\frac{C}{\lambda}\left[2K+1\right]\|f\|_1,
\end{align*}
and also
\begin{align*}
\mu\times\sharp{(\bar{E})}&\geq\sum_{j=-K}^K\mu\left\{x:G_j^K(x)>\lambda\right\}\\
&\geq\sum_{j=-K}^K\mu\left\{x:G_j(x)>\lambda\right\}\\
&=\sum_{j=-K}^K\mu{\left(\tau^{-j}E\right)}\\
&=(2K+1)\mu{(E)}.
\end{align*}
Thus, we have
$$\sum_{n=1}^\infty\mu (E)\leq\frac{C}{\lambda}\|f\|_1$$
and this completes our proof.
\end{proof}

\end{document}